\documentclass[11pt,reqno]{amsproc}
\usepackage[margin=1in]{geometry}
\usepackage{amsmath, amsthm, amssymb, amsrefs}
\usepackage{times, esint,stackrel}
\usepackage{enumitem}
\usepackage{color}  
\usepackage[colorlinks=true]{hyperref}
\usepackage[color=yellow]{todonotes}
\usepackage{mathtools}
\usepackage{etoolbox}
\usepackage{mathrsfs}
\usepackage{accents}
\usepackage{graphicx}
\usepackage{framed,xfrac}

\makeatletter
\patchcmd\@thm
  {\let\thm@indent\indent}{\let\thm@indent\noindent}%
  {}{}
\makeatother

\usepackage{etoolbox}
\expandafter\patchcmd\csname\string\proof\endcsname
  {\normalparindent}{0pt }{}{}

\newtheorem*{thm*}{Theorem}

\newcommand{\be}{\begin{equation}}
\newcommand{\ee}{\end{equation}}
\newcommand{\bea}{\begin{eqnarray}}
\newcommand{\eea}{\end{eqnarray}}

\newtheorem{thm}{Theorem}
\newtheorem{prop}{Proposition}
\newtheorem{defn}{Definition}
\newtheorem{lemma}{Lemma}

\newtheorem{example}{Example}
\theoremstyle{definition}

\newcommand{\rmd}{{\rm d}}

\newcommand{\ol}[1]{\mkern 1.5mu\overline{\mkern-1.5mu#1\mkern-1.5mu}\mkern 1.5mu}

\def\e{{\rm e}}

\newcommand{\bq}{\begin{equation}}
\newcommand{\eq}{\end{equation}}
\newcommand{\bqa}{\begin{eqnarray*}}
\newcommand{\eqa}{\end{eqnarray*}}

\def\XXint#1#2#3{{\setbox0=\hbox{$#1{#2#3}{\int}$ }
\vcenter{\hbox{$#2#3$ }}\kern-.6\wd0}}

\title[Statistically self-similar mixing by Gaussian random fields]{Statistically self-similar mixing by Gaussian random fields}

 \author[M. Coti Zelati]{Michele Coti Zelati}
\address{Department of Mathematics, Imperial College London, London, SW7 2AZ, UK}
\email{m.coti-zelati@imperial.ac.uk}

\author[T. D. Drivas]{Theodore D. Drivas}
\address{Department of Mathematics, Stony Brook University, Stony Brook, NY, 11790}
\email{tdrivas@math.stonybrook.edu}

\author[R. S. Gvalani]{Rishabh S. Gvalani}
\address{Max-Planck-Institut für Mathematik in den Naturwissenschaften, Inselstraße 22, 04103 Leipzig}
\email{gvalani@mis.mpg.de}

\date{today}

\begin{document}
\begin{abstract}
We study the passive transport of a scalar field by a spatially smooth but white-in-time incompressible Gaussian random velocity field on $\mathbb{R}^d$.  If the velocity field $u$ is homogeneous, isotropic, and statistically self-similar, we derive an exact formula which captures non-diffusive mixing.  For zero diffusivity, the formula takes the shape of $\mathbb{E}\  \| \theta_t \|_{\dot{H}^{-s}}^2 = \e^{-\lambda_{d,s} t}  \| \theta_0 \|_{\dot{H}^{-s}}^2$  with any $s\in (0,d/2)$ and  $\frac{\lambda_{d,s}}{D_1}:= s(\frac{\lambda_{1}}{D_1}-2s)$ where $\lambda_1/D_1 = d$  is the top Lyapunov exponent  associated to the random  Lagrangian flow generated by $u$ and $ D_1$ is small-scale shear rate of the velocity.  Moreover, the mixing is shown to hold \emph{uniformly} in diffusivity.
\end{abstract}

\maketitle

\vspace{-2mm}

\section{Passive scalar transport by  Gaussian random fields}

We study passive scalar $\theta_t^\kappa(x) :\mathbb{R}^+\times \mathbb{R}^d\to \mathbb{R}$ transport with diffusivity  $\kappa\geq 0$ on $(t,x)\in \mathbb{R}^+\times \mathbb{R}^d$
\begin{align}\label{thetaeq}
\rmd \theta_t^\kappa +\rmd u_t \circ \nabla \theta_t^\kappa &= \kappa\Delta \theta_t^\kappa\ \rmd t+ \rmd f_t ,\\
\theta_t^\kappa|_{t=0}&= \theta_0,\\
\int_{\mathbb{R}^d} \theta_0\rmd x &= 0,
\end{align}
 where $ u_t:=u(x,t)$ is a white-in-time, incompressible Gaussian random field with mean and covariance
\begin{align}
\mathbb{E}\left[ {u}^i(x,t)\right] &=0,\\
\mathbb{E}\left[ {u}^i(x,t)  {u}^j(x',t')\right] &= D^{ij} (x,x') \delta (t-t'),
\end{align}
We shall consider random fields which are homogeneous and divergence-free 
\begin{align}
D^{ij} (x,x')  &:= D^{ij} (x-x'), \\
\partial_{x_i} D^{ij} (x,x') &= \partial_{x_j} D^{ij} (x,x') =0.
\end{align}
For visualization of  a Gaussian velocity whose covariance mimics inertial range turbulence, see Figure \ref{fig1}.
\begin{figure}
  \includegraphics[width=0.32\linewidth]{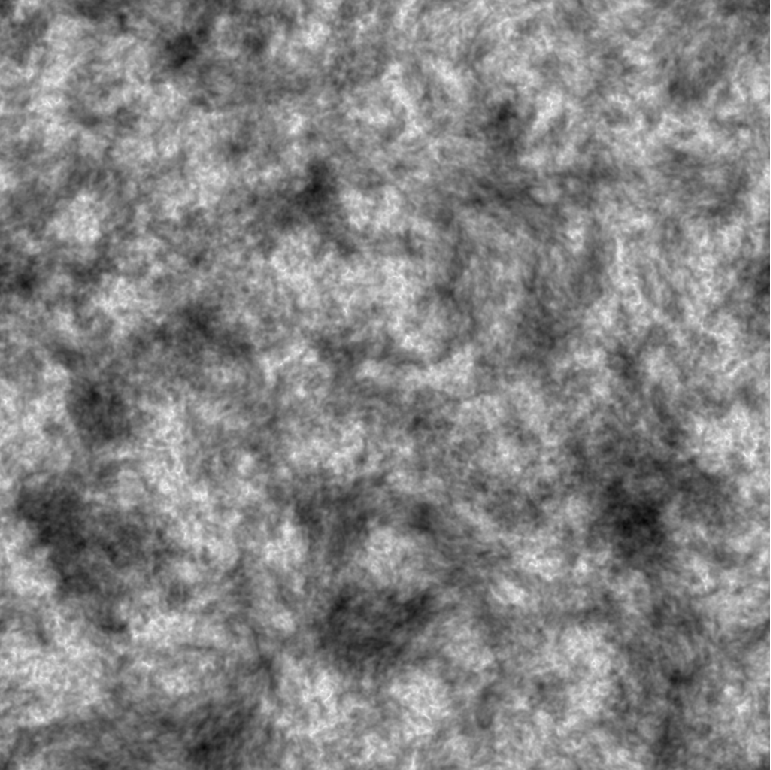}\hspace{1.5mm}
    \includegraphics[width=0.32\linewidth]{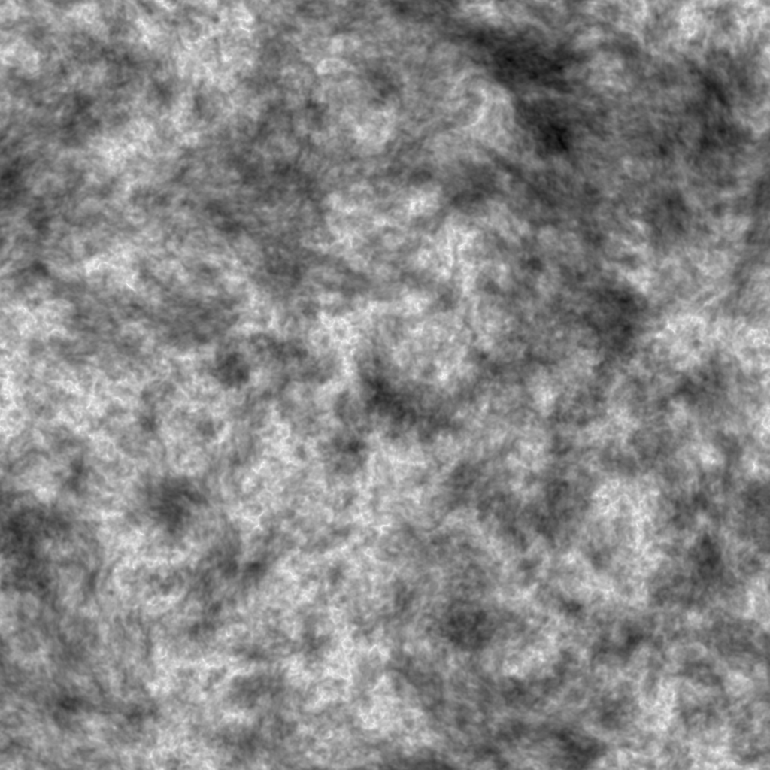} \hspace{0.5mm}
      \includegraphics[width=0.32\linewidth]{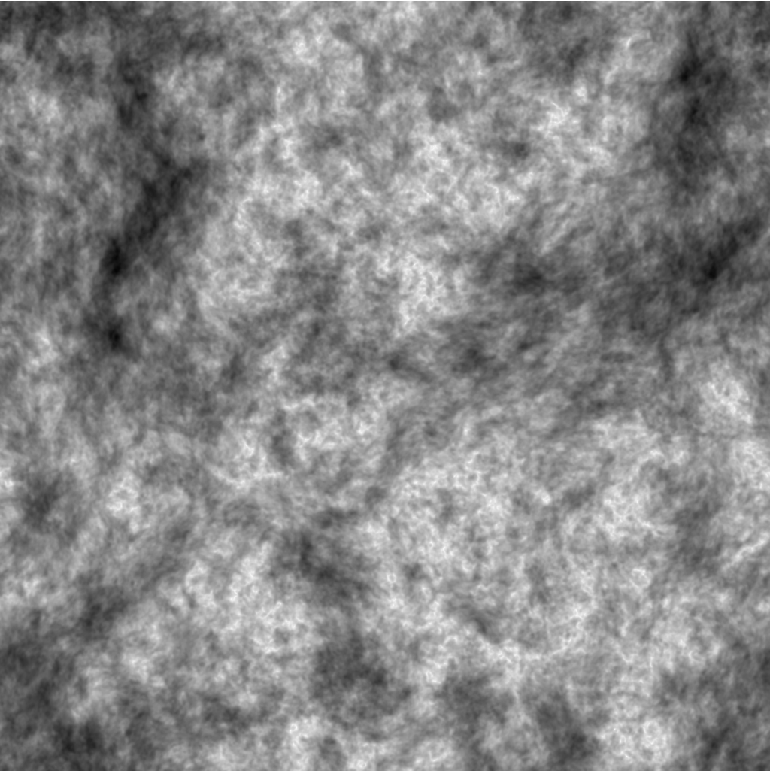}
  \caption{Snapshots in space at various times of a white-in-time Gaussian random velocity.}
  \label{fig1}
\end{figure}

The  homogenous Gaussian  forcing $f_t := f(x,t)$ is taken independent  and defined by the covariance
\be
\mathbb{E}\left[ {f}(x,t)  {f}(x',t')\right] = F (x-x') \delta (t-t'). 
\ee
Equation \eqref{thetaeq} is interpreted in the Stratonovich sense (denoted by $\circ$). The  It\^{o} form of equation  \eqref{thetaeq}  reads 
\begin{align}\label{thetaeqito}
\rmd \theta_t^\kappa +\rmd u_t\cdot \nabla \theta_t^\kappa &=\tfrac{1}{2} D(0) :\nabla \otimes \nabla \theta_t^\kappa \rmd t + \kappa\Delta \theta_t^\kappa\ \rmd t+ \rmd f_t ,\\
\theta_t^\kappa|_{t=0}&= \theta_0.
\end{align}
The setup is called the Kraichnan model \cite{Kraichnan68,FGV01, CFG08}, reviewed in the wonderful lecture notes of Gaw\c{e}dzki \cite{gawedzki1997turbulence,gawedzki1999easy,gawedzki2000soluble}.
The velocity  field can be written concretely as
\be\label{traj}
\rmd u_t(x)=  \sum_{n=0}^\infty\sigma^{(n)}(x) \rmd W^{(n)}_t, \qquad \sigma^{(n)}(x) := \sqrt{\lambda_n} e_n(x),
\ee
 where $W^{(n)}_t$ are independent one-dimensional  Brownian motions and where $\lambda_n$ and $e_n(x)$ for $n=0,1,2,\dots$ are the eigenvalues and eigenfunctions of the positive, trace-class operator with kernel $D (x-x')$ acting on $L^2(\mathbb{R}^d, \mathbb{R}^d)$.  The kernel $D$ can be represented by
\be\label{dform}
D^{ij}(x-x')=\sum_{n=0}^\infty \sigma^{(n)}_i(x) \sigma^{(n)}_j(x').
\ee
  Here the $\sigma^{(n)}$ are divergence-free vector fields since $D$ is divergence-free in each index.  We will consider a model of the viscous--convective range for the scalar in which the velocity if Lipschitz, the so-called Batchelor regime for the model.  The rough version of the model, mimicking a turbulent inertial range, has been a topic of a great deal of study \cite{BGK98,LJR02,LJR04}.  See also discussion in \cite{eyink2022kraichnan,eyink2022high} regarding the effects of molecular fluctuations on the phenomenology of scalar advection below the viscous--convective range.  In this work, we will specialize our covariance even further
  
\begin{defn}[Self-similar isotropic covariance]
We say $D$ is self-similarly isotropic if it takes the form
\begin{align}\label{asymexpD2}
 D(0)-D(r) 
&=D_1\left[I+ \left( \frac{2}{d-1}\right)\left(I- \hat{r}\otimes \hat{r}\right)\right]|r|^2
\end{align}
for some constant\footnote{The physical dimensions of $D_1$ are inverse time and it can be regarded as a proxy for the shear rate 
at small scales. See the discussion of Kraichnan, e.g.  \cite[Eq. (3.5)]{Kraichnan68}.} $D_1>0$ and  $D(0)= D_0I$ for $D_0>0$.
\label{def:self-similar}
\end{defn}

\begin{example}
A prototypical example of a self-similarly isotropic random field arises in the following context.
Consider a Gaussian field defined by the covariance 
\be
D_{ij}(r) = \bar{D}_0 \int  \frac{P_{ij}(k)}{(|k|+m^2)^{(d-\zeta)/2}} \e^{i k\cdot r} \rmd k 
\ee
where $\zeta>2$. The constant $m$ is an infrared  cutoff for the  velocity and $P_{ij}$ is  the projection onto the divergence-free subspace.    At short distances, one can compute that
\be
D_0\delta^{ij}-D^{ij}(r) 
=D_1\left[\delta^{ij}+ \left( \frac{2}{d-1}\right)\left(\delta^{ij}- \hat{r}^i\hat{r}^j\right)\right]r^2 + O((m|r|)^2),
\ee
where 
\be\label{D0D1const}
{D}_0 =  \frac{\ol{D}_0}{m^2 } \frac{d-1}{d (4\pi)^{d/2}  \Gamma \left(\frac{d+2}{2}\right)} , \qquad {D}_1 =  \ol{D}_0 \frac{d-1}{(d+2)(4\pi)^{d/2}  \Gamma \left(\frac{d+2}{2}\right) } .
\ee
Thus, taking the infrared limit $m\to 0$ gives a self-similarly isotropic random field.
\end{example}

    The forcing in \eqref{thetaeq} is also represented by a statistically independent collection of one-dimensional independent Brownian motions $\{B^{(k)}_t\}_{k\in \mathbb{N}}$ and of scalar functions $\{ q^{(k)}\}_{k\in \mathbb{N}}$,
as 
\be
\rmd f_t(x)= \sum_{k=0}^\infty q^{(k)} (x)  \rmd B^{(k)}_t, \qquad q^{(k)}(x) := \sqrt{\mu_k} u_k(x),
\ee
 where $B^{(k)}_t$ are independent Brownian motions and where $\mu_k$ and $u_k(x)$ for $k=0,1,2,\dots$ are the eigenvalues and eigenfunctions of the positive, trace-class operator with kernel $F (x-x')$ acting on $L^2(\mathbb{R}^d, \mathbb{R})$.
We denote averages over the random velocity by $\mathbb{E}^u[\cdot]$, averages over the forcing by $\mathbb{E}^f[\cdot]$  and averages over both simply by $\mathbb{E}[\cdot]$.

In this note, we are interested in the mixing properties of Gaussian random fields. In the context of stochastic transport, exponential mixing estimates have been obtained in  \cites{BBPS22,BBPS21} for velocities generated by the stochastic Navier-Stokes equations, 
in  \cite{GY21} for Kraichnan-type models with noise satisfying a general Hörmander condition, and in \cites{BCZG23,Cooperman22} for alternating shear flows with either random phases or random switching times. On the other hand, deterministic constructions of exponentially mixing flows can be found in \cites{ELM23,ACM19,YZ17,EZ19}. The purpose of this note is to give a new and simple, yet explicit and quantitative  proof of mixing for \eqref{thetaeq}. 
Our main result, which quantifies mixing in terms of negative Sobolev norms \cite{oakley2021mix}, is

\begin{thm}[Scalar Mixing Identity]\label{thm1}
Suppose that  $D$ is an incompressible, homogeneous and self-similarly isotropic correlation function  \eqref{asymexpD2}.
Then we have the identity
\be\label{eq:mixidentity}
\mathbb{E}  \| \theta^\kappa_t \|_{\dot{H}^{-s}}^2 = \e^{-\lambda_{d,s} t}  \| \theta_0 \|_{\dot{H}^{-s}}^2 + \frac{F_s}{\lambda_{d,s}}
 \Big(1-  \e^{-\lambda_{d,s} t}\Big) -  \kappa \int_0^t \e^{-\lambda_{d,s}(t-\tau)}  \mathbb{E} \| \theta_\tau^\kappa \|_{\dot{H}^{1-s}}^2  \rmd \tau
\ee
where $ {\lambda}_{d,s}: =  2D_1s(d-2s)$ and $F_s=  \sum_{k=0}^\infty \| q^{(k)} \|_{\dot{H}^{-s}}^2$. 
  \end{thm}

In particular, when there is no forcing ($q^{(k)}\equiv 0$), we obtain mixing in expectation $\mathbb{E}\  \| \theta_t \|_{H^{-s}}^2 \to 0$ with an exponentially fast rate, with rate uniform in the diffusivity $\kappa$.
The identity \eqref{eq:mixidentity} may look surprising, since mixing estimates typically require to pay derivatives on the initial datum to gain decay. This would indeed be the case for a pathwise estimate, since in general solutions can mix and unmix and hence a time reversal in a $\dot{H}^{-s}\mapsto \dot{H}^{-s}$ would be in conflict with a decay estimate.
The point is that  \eqref{eq:mixidentity}  is an identity \emph{in expectation}, hence the possible realizations in which growth and unmixing happen are averaged out.

\section{Adapted Mixing Identity and Proof of Theorem \ref{thm1}}

We establishes an exact balance of ``mixing norms" tailored in a particular way to the  covariance of the velocity field $D$.  A similar identity appeared recently in the work of Coghi and Maurelli, where it was used to improve the local wellposedness theory for stochastic Euler with multiplicative noise \cite{coghi2023existence}.
\begin{lemma}[Adapted Mixing Identity]\label{keylem}
Let $D$ be homogenous and divergence-free. Suppose that a function $G:\mathbb{R}^d\to\mathbb{R}$ can be found to satisfy
\be\label{Geqn}
\big(D(0)- D(r)\big) : \nabla_r \otimes \nabla_r  G = -\lambda G,
\ee
pointwise in space
for some $\lambda\in \mathbb{R}$.
Then the solution $\theta_t^\kappa$ of \eqref{thetaeq} satisfies the following ``mixing" identity,
\be
\mathbb{E}\  \langle \theta_t^\kappa, G* \theta_t^\kappa\rangle_{L^2} = \e^{-\lambda t}\langle \theta_0, G* \theta_0\rangle_{L^2} + \frac{1}{\lambda}F_G\Big(1-  \e^{-\lambda t}\Big)- \kappa \int_0^t \e^{-\lambda(t-\tau)}  \mathbb{E}\langle \nabla \theta_\tau^\kappa,   G* \nabla\theta_\tau^\kappa\rangle_{L^2} \rmd \tau
\ee
where the source is $F_G:=2\int_{\mathbb{R}^d} G(r) F(r)\rmd r  $.
  \end{lemma}

\begin{proof}[Proof of Lemma \ref{keylem}]
First, by definition, we have
\be
 \langle \theta_t^\kappa, G* \theta_t^\kappa\rangle_{L^2} =   \int_{\mathbb{R}^d}  \theta_t^\kappa(x) \int_{\mathbb{R}^d}  G(r)  \theta_t^\kappa(x+r) \rmd r \rmd x.
\ee
Recall (leaving the sum over $n$ implicit) that
\begin{align*}
\rmd \theta_t^\kappa +  \sigma^{(n)} \cdot \nabla \theta_t^\kappa \rmd W^{(n)}_t&=\tfrac{1}{2} D(0) :\nabla \otimes \nabla \theta_t^\kappa \rmd t +\kappa\Delta \theta_t^\kappa \rmd t + q^{(n)} \rmd B^{(n)}_t,\\
\rmd (G* \theta_t^\kappa) + G* (\sigma^{(n)}\cdot \nabla\theta_t^\kappa)  \rmd W^{(n)}_t&=\tfrac{1}{2} D(0) :\nabla \otimes \nabla (G* \theta_t^\kappa) \rmd t +\kappa\Delta (G* \theta_t^\kappa) \rmd t+ (G*q^{(n)})\rmd B^{(n)}_t.
\end{align*}
Thus, by Ito's product rule, we have
\begin{align*}
\rmd  \left( \theta_t^\kappa \  G* \theta_t^\kappa\right)  &=   (G* \theta_t^\kappa) \  \rmd \theta_t^\kappa  +     \theta_t^\kappa\ \rmd(G* \theta_t^\kappa) + \rmd[   \theta_t^\kappa, G* \theta_t^\kappa]   \\
&=-  (G* \theta_t^\kappa ) \sigma^{(n)} \cdot \nabla \theta_t^\kappa \rmd W^{(n)}_t - \theta_t^\kappa (G* (\sigma^{(n)}\cdot \nabla\theta_t^\kappa) ) \rmd W^{(n)}_t\\
&\qquad + (G* \theta_t^\kappa ) q^{(n)} \rmd B^{(n)}_t + \theta_t^\kappa (G* q^{(n)} ) \rmd B^{(n)}_t\\
&\qquad \ \ + \tfrac{1}{2} (G* \theta_t^\kappa )   D(0) :\nabla \otimes \nabla \theta_t^\kappa \rmd t + \tfrac{1}{2}\theta_t^\kappa D(0) :\nabla \otimes \nabla (G* \theta_t^\kappa) \rmd t\\
&\qquad\quad + \kappa \theta_t^\kappa \Delta (G* \theta_t^\kappa)+ \kappa (G* \theta_t^\kappa)\Delta \theta_t^\kappa    +  \sigma^{(n)} \cdot \nabla \theta_t^\kappa \ (G* (\sigma^{(n)}\cdot \nabla\theta_t^\kappa)) \rmd t.
\end{align*}
Upon integrating over space and using the fact that mollification is self-adjoint in $L^2$, we have
\begin{align*}
\rmd  \langle \theta_t^\kappa, G* \theta_t^\kappa\rangle_{L^2}  &=- 2  \langle G* \theta_t^\kappa,  \sigma^{(n)} \cdot \nabla \theta_t^\kappa\rangle_{L^2}  \rmd W^{(n)}_t + 2 \langle\theta_t^\kappa ,(G* q^{(n)} )\rangle_{L^2} \rmd B^{(n)}_t\\\
&\qquad + (D(0)+ \kappa I) : \langle \theta_t^\kappa, \nabla \otimes \nabla (G* \theta_t^\kappa)\rangle_{L^2}  \\
&\qquad\quad-  \int_{\mathbb{R}^d} \int_{\mathbb{R}^d}  \theta_t^\kappa(x)  \sigma^{(n)}(x)\otimes \sigma^{(n)}(x+r) :  \nabla \otimes \nabla (G* \theta_t^\kappa) \rmd x\rmd r \rmd t\\
&\qquad\quad\quad  + 2\int_{\mathbb{R}^d}  \int_{\mathbb{R}^d} G(r) F(x,x+r)\rmd r  \rmd x  \\ 
 &=- 2  \langle G* \theta_t^\kappa,  \sigma^{(n)} \cdot \nabla \theta_t^\kappa\rangle_{L^2}  \rmd W^{(n)}_t+ 2 \langle\theta_t^\kappa ,(G* q^{(n)} )\rangle_{L^2} \rmd B^{(n)}_t \\
&\qquad +  \int_{\mathbb{R}^d} (D(0)- D(r)+ \kappa I) : \langle \theta_t^\kappa(\cdot), G(r)  (\nabla \otimes \nabla \theta_t^\kappa)(\cdot + r) \rangle_{L^2} \rmd r\rmd t \\
&\qquad\quad\quad  +2 \int_{\mathbb{R}^d}  \int_{\mathbb{R}^d} G(r) F(x,x+r)\rmd r  \rmd x.
\end{align*}
Thus, we obtain
\begin{align*}
\rmd  \langle \theta_t^\kappa, G* \theta_t^\kappa\rangle_{L^2} 
 &=- 2  \langle G* \theta_t^\kappa,  \sigma^{(n)} \cdot \nabla \theta_t^\kappa\rangle_{L^2}  \rmd W^{(n)}_t + 2 \langle\theta_t^\kappa ,(G* q^{(n)} )\rangle_{L^2} \rmd B^{(n)}_t\\
&\qquad +  \int_{\mathbb{R}^d} (D(0)- D(r)+ \kappa I) :(\nabla_r \otimes \nabla_r G)(r)  \langle \theta_t^\kappa(\cdot), \theta_t^\kappa(\cdot + r) \rangle_{L^2} \rmd r\rmd t\\
&\qquad\quad\quad  + 2\int_{\mathbb{R}^d}  \int_{\mathbb{R}^d} G(r) F(x,x+r)\rmd r  \rmd x.
\end{align*}
In the above calculation, we repeatedly made use of the fact that the $\sigma^{(n)}$ are divergence-free, integrated by parts and changed some $x$ to $r$ derivatives. 
Appealing to the defining equation  \eqref{Geqn} for the kernel $G$, we find 
\begin{align}\nonumber
\rmd  \langle \theta_t^\kappa, G* \theta_t^\kappa\rangle_{L^2}  &= -\lambda \langle \theta_t^\kappa, G* \theta_t^\kappa\rangle_{L^2}   - 2  \langle G* \theta_t^\kappa,  \sigma^{(n)} \cdot \nabla \theta_t^\kappa\rangle_{L^2}  \rmd W^{(n)}_t+ 2 \langle\theta_t^\kappa ,(G* q^{(n)} )\rangle_{L^2} \rmd B^{(n)}_t+ F_G\\
&\qquad + \kappa \langle \Delta \theta_t^\kappa,  G* \theta_t^\kappa\rangle_{L^2}, \label{eq:pathwiseH-s}
\end{align}
where we have introduced the notation $F_G$ from the statement of the lemma. The result follows upon taking expectation, using the mean zero property of the martingale term and subsequently integrating.
\end{proof}

We now demonstrate a solution $G$ to \eqref{Geqn}  if it covariance function is self-similarly isotropic.

\begin{lemma}\label{isolem}
Suppose that  $D$ is self-similarly isotropic. Then for any $s\in (0,d/2)$, the Riesz potential 
\begin{equation}\label{Riesz}
G(r)= I_s(|r|), \qquad  I_s(\rho)=\frac{1}{c_{d,s}} \frac{1}{\rho^{d-2s}}, \qquad c_{d,s}=\pi^{d/2}2^{2s} \frac{\Gamma(s)}{\Gamma((d-2s)/2)}
\end{equation}
solves equation  \eqref{Geqn}   with $\lambda=  2D_1s(d-2s)$. 
\end{lemma}

\begin{proof}[Proof of Lemma \ref{isolem}]  We will prove something more general here by studying instead the covariance
\begin{align}\label{asymexpD}
 D(0)-D(r) 
&=D_1\left[I+ \left( \frac{\zeta}{d-1}\right)\left(I- \hat{r}\otimes \hat{r}\right)\right]|r|^\zeta
\end{align}
with the statement of the lemma following as a corollary with $\zeta=2$.  Let $z=|r|$. 
We seek a radial kernel $G(r):= \mathcal{G}(z)$. On such a function, we have
\begin{align*}
\nabla_r\otimes \nabla_r  {G}&= \nabla_r \otimes (\hat{r}  \mathcal{G}'(z) )= \frac{1}{z} ( I - \hat{r}\otimes \hat{r}) \mathcal{G}'(z)  +  \hat{r}\otimes \hat{r}\mathcal{G}''(z).
\end{align*}
Note that $( I - \hat{r}\otimes \hat{r}) :I=( I - \hat{r}\otimes \hat{r}) :( I - \hat{r}\otimes \hat{r}) =d- 1$ so
\begin{align*}
\big(D(0)- D(r)\big) :( I - \hat{r}\otimes \hat{r})&= (d-1 + \zeta )D_1 |r|^\zeta,  \\
\big(D(0)- D(r)\big) : \hat{r}\otimes \hat{r} &= D_1 |r|^\zeta.
\end{align*}
Thus we have
\begin{align*}
\big(D(0)- D(r)  \big) : \nabla_r \otimes \nabla_r  \mathcal{G}&=  D_1 \Big[ ((d-1) + \zeta)z\mathcal{G}'(z)+  z^2\mathcal{G}''(z)\Big]z^{\zeta-2}.
\end{align*}
Thus, equating this to $ -\lambda \mathcal{G}$, we find that the kernel must solve
\be\label{isoGeqn}
z^2\mathcal{G}''(z)+ ((d-1) + \zeta)z \mathcal{G}'(z)+ \frac{\lambda}{D_1}z^{2-\zeta}\mathcal{G}(z)= 0.
\ee 
We seek a solution of the type $\mathcal{G}(z)= z^{\alpha}\ln^\beta(z)$.  Demanding the ansatz be a solution, we require
\begin{align*}
0&= \Bigg[ \alpha^2  + \Big(d+\zeta-2 \Big) \alpha + \frac{\lambda}{D_1}z^{2-\zeta}\Bigg]\ln^2(z) +\beta \Big(d+\zeta-2 + 2\alpha\Big)\ln(z) +  \beta(\beta-1) .
\end{align*}
Note that for this ansatz to be a solution, we require $\beta=0$ or $\beta=1$. Thus, we have the cases
\begin{enumerate}
\item[(a)] If $\lambda=0$, then 
\be
\alpha = 2-d-\zeta, \qquad  \beta=0,
\ee
giving a statistically conserved quantity.
\item[(b)]  If  $\beta=0$ and $\lambda\neq 0$ then 
\be
\zeta=2, \qquad \alpha =-\frac{d}{2}   \pm \sqrt{\left(\frac{d}{2} \right)^2- \frac{\lambda}{D_1}},
\ee
provided that $\tfrac{\lambda}{D_1}\leq \left(\frac{d}{2}\right)^2$. 
\item[(c)]  If $\beta=1$  then 
\be
\zeta=2, \qquad \alpha =-\frac{d}{2} , \qquad \frac{\lambda}{D_1}=\left(\frac{d}{2} \right)^2.
\ee
\end{enumerate} 
The stated conclusion follows easily as a special case of (b).
\end{proof}

Recall the following characterization of the $H^{-s}$ semi-norm

\begin{lemma}[\cites{Silvestre07,Landkof72}]\label{Hm1lem}
Let $d\geq 2$, $s\in (0,d/2)$. Then $ \langle h, I_s* h\rangle_{L^2} = \|h \|_{\dot{H}^{-s}}^2$ where  $I_s$ is the Riesz potential.
\end{lemma}

With this, Theorem \ref{thm1} follows  immediately from Lemmas \ref{keylem}, \ref{isolem} and \ref{Hm1lem}.

\subsection{Two-particle dispersion}
We now show that similar arguments as those used in Lemma~\ref{keylem} can be used to prove a kind of decay estimate for the law of the two-point Lagrangian flow associated to~\eqref{thetaeq}, i.e.
\begin{align}
\rmd X_t(x) = \sum_{n=0}^\infty \sigma^{(n)}(X_t(x)) \circ  \rmd W_t^{(n)} \qquad 
\rmd X_t(y)= \sum_{n=0}^\infty \sigma^{(n)}(X_t(y)) \circ  \rmd W_t^{(n)} \, ,
\end{align}
with $X_0(x)=x,X_0(y)=y$. Since the velocity field is spatially homogeneous, we need only study the law $\mu_t= \mathrm{Law}(R_t({r}))$ of the separation between particles $R_t({r})=X_t(x)-X_t({y})$ with $r=x-y$. Using the spatial homogeneity of the velocity field $u$, one can check that $\mu_t$ is a solution of the following PDE
\begin{equation}
\partial_t \mu_t = \left(D(0)-D(r)\right):\nabla_r \otimes \nabla_r \mu_t \, ,
\label{2pointFP}
\end{equation}
with $\mu_0=\delta_r$. Given the above characterisation of $\mu_t$, we have the following result.
\begin{prop}
Assume there exists a $G$ as in the statement of Lemma~\ref{keylem}. Then, we have 
\begin{equation}
\int_{\mathbb{R}^d} G(\bar r) \, \mathrm{d}\mu_t(\bar r) = \e^{-\lambda t} G(r) \, .
\end{equation}
\end{prop}
\begin{proof}
The proof follows directly by testing~\eqref{2pointFP} with $G$, integrating by parts using the fact that $D$ is divergence-free, and then using the definition of $G$.
\end{proof}
If $D$ is self-similarly isotropic, in view of Lemma \ref{isolem}, this means that for any $s\in (0,d/2)$, we have
\be
\mathbb{E}[|X_t(x) - X_t(y)|^{2s-d}] = \e^{-\lambda_{d,s} t }|x-y|^{2s-d}, \qquad \lambda_{d,s}/D_1 = 2s(d-2s).
\ee
Thus, we obtain exact decay laws for (inverse) particle dispersion, capturing the average dispersion.  In the case $s=0$, this relation becomes the statistical conservation law  $\mathbb{E}[|X_t(x) - X_t(y)|^{-d}] =|x-y|^{-d}$ derived in the works \cite{falkovich2013single,zel1984kinematic}. This fact can be understood as akin to the statement that in a given, \emph{linear, isotropic and incompressible} velocity field, the average $\fint_{S^{d-1}} |X_t(x) - X_t(x+r)|^{-d}\rmd \omega(\hat{r})$ is a constant of motion. See the work of Frishman et al \cite{frishman2015statistical}.

\section{Top Lyapunov exponent in the Kraichnan model}
In this section, we present explicitly computed the Lyapunov exponents of the Lagrangian flow associated to~\eqref{thetaeq}. This expression for the Lyapunov exponents for first derived by Le Jan in~\cite{LJ85} in his study of  isotropic Brownian flows. For completeness, we present this calculation for the leading exponent below.
\begin{prop}
Consider the Lagrangian flow $\varphi_t(x)$ associated to~\eqref{thetaeq}, i.e. the following Stratonovich SDE
\begin{equation}\label{floweqn}
\rmd \varphi_t(x) = \sum_{n=0}^\infty \sigma^{(n)}(\varphi_t(x))\circ \rmd W_t^{(n)}\, ,
\end{equation}
with $\varphi_t(x)=x$. Furthermore, assume that $D$ is self-similarly isotropic in the sense of Definition~\ref{def:self-similar}. Then
\begin{equation}\label{lyapexp}
\log (\lVert D \varphi_t \rVert) = d D_1  t +  M_t ,
\end{equation}
where $M_t$ is the martingale defined by \eqref{Mt}.  The top Lyapunov exponent, defined almost surely, is 
\be
\lambda_1 =\lim_{t\to\infty} \frac{\log (\lVert D \varphi_t \rVert) }{t}= d D_1.
\label{eq:Lyapunov}
\ee
\end{prop}

\begin{proof}
Let us define by $A_{t,x}:= (D\varphi_t)(x)$. It is straightforward to check that
\begin{equation}\label{Aeqn}
\rmd A_{t,x} = \sum_{n=0}^\infty (D \sigma^{(n)})(\varphi_t(x))\cdot A_{t,x}  \rmd W_t^{(n)}\,. 
\end{equation}
Indeed, since $\sigma^{(n)}$ are divergence-free, the It\^o  and Stratonovich form of SDE \eqref{floweqn} coincide.
Pick $v \in \mathbb{S}^{d-1}$ and define $v_{t,x}:= A_{t,x}v$. We then have that
\begin{equation}
\rmd v_{t,x} =  \sum_{n=0}^\infty (D \sigma^{(n)})(\varphi_t(x))\cdot v_{t,x} \, \rmd W_t^{(n)} \,.
\end{equation}
Applying It\^o's formula, we obtain
\begin{align}
\rmd  |v_{t,x}|^2 = \sum_{n=0}^\infty |(D \sigma^{(n)})(\varphi_t(x))\cdot v_{t,x}|^2 \, dt + 2 \sum_{n=0}^\infty\langle v_{t,x}, (D \sigma^{(n)})(\varphi_t(x))\cdot v_{t,x} \rangle \, \rmd W_t^{(n)} \, .
\label{eq:intermediateIto1}
\end{align}
Note that the drift term in the above expression can be simplified as follows
\begin{align}\nonumber
\sum_{n=0}^\infty |(D \sigma^{(n)})(\varphi_t(x))\cdot v_{t,x}|^2 =& \sum_{n=0}^\infty\sum_{i=1}^d \left|\sum_{j=1}^d \partial_{j} \sigma^{(n)}_i v_{t,x}^j\right|^2 =  \sum_{n=0}^\infty \sum_{i=1}^d\sum_{j,k=1}^d \partial_{j} \sigma^{(n)}_i \partial_{k} \sigma^{(n)}_i v_{t,x}^j v_{t,x}^k \\
=& \langle v_{t,x}, C \cdot v_{t,x} \rangle \, ,
\end{align}
with the matrix $C$ given by
$
C_{jk}(x) :=\sum_{i=1}^d \left.\partial_{x_{j}}\partial_{x_{k}'}D^{ii}(x-x') \right|_{x=x'} \, .
$
Using the expression for $D$ given by~\eqref{asymexpD2}, we obtain
$
\partial_{x_{j}}\partial_{x_{k}'}D^{ii}(x-x') = 2D_1 \delta_{jk}\left(\frac{d+1}{d-1} -\frac{2}{d-1}\delta_{ik} \right) \, .
$
Using the above,~\eqref{eq:intermediateIto1} reduces to 
\begin{align}
\rmd |v_{t,x}|^2 = 2D_1(d+2)|v_{t,x}|^2 \, dt + 2 \sum_{n=0}^\infty\langle v_{t,x}, (D \sigma^{(n)})(\varphi_t(x))\cdot v_{t,x} \rangle \, \rmd W_t^{(n)} \, .
\end{align}
Applying It\^o's formula to the above expression again, we arrive at
\begin{align}
\rmd \log |v_{t_x}| = &\frac12 \rmd\log |v_{t,x}|^2= D_1(d+2) \, dt  - \frac{1}{ |v_{t,x}|^4} \sum_{n=0}^\infty \left|\langle v_{t,x}, (D \sigma^{(n)})(\varphi_t(x))\cdot v_{t,x} \rangle\right| ^2 \, dt +dM_t \, ,
\label{eq:intermediateIto2}
\end{align}
where the martingale term $M_t$ is given by
\begin{equation}\label{Mt}
M_t =\int_0^t\frac{1}{|v_{t,x}|^2}  \sum_{n=0}^\infty\langle v_{t,x}, (D \sigma^{(n)})(\varphi_t(x))\cdot v_{t,x} \rangle \, \rmd W_t^{(n)} \, .
\end{equation}
We first simplify the drift term in~\eqref{eq:intermediateIto2} as follows
\begin{align}\nonumber
\sum_{n=0}^\infty\left|\langle v_{t,x}, (D \sigma^{(n)})(\varphi_t(x))\cdot v_{t,x} \rangle\right|^2 &=\sum_{n=0}^\infty \left|  \sum_{j,k=1^d} v_{t,x}^k \partial_{j}\sigma^{(n)}_k v_{t,x}^j  \right|^2 = \sum_{n=0}^\infty \sum_{j,k,\ell,m=1}^d  v_{t,x}^k v_{t,x}^\ell \partial_{j}\sigma^{(n)}_k \partial_{\ell}\sigma^{(n)}_m v_{t,x}^j v_{t,x}^m \, .
\end{align}
Note now that
$
\sum_{n=0}^\infty\partial_{j}\sigma^{(n)}_k \partial_{\ell}\sigma^{(n)}_m = \left.\partial_{x_j}\partial_{x_\ell'}D^{k m}(x-x' )\right|_{x=x'}\, .$
Using~\eqref{asymexpD2} again, we obtain
\begin{align}
\partial_{x_j}\partial_{x_\ell'}D^{k m}(x-x' ) =& 2D_1\left(\delta_{k m} \delta_{j \ell } \left(\frac{d+1}{d-1} \right) -\frac{1}{d-1}(\delta_{kj}\delta_{m \ell} + \delta_{m j }\delta_{k \ell}) \right) \, .
\end{align}
From~\eqref{eq:intermediateIto2}, we obtain
$
\rmd \log |v_{t_x}|  = D_1 d \, \rmd t + \rmd M_t \, ,
$
whence \eqref{lyapexp} follows. For~\eqref{eq:Lyapunov}, we note that
\begin{align}
\mathbb{E}(M_t^2) = \sum_{n=0}^\infty \int_0^t \mathbb{E}\left( \frac{1}{|v_{t,x}|^4}  |\langle v_{t,x}, (D \sigma^{(n)})(\varphi_t(x))\cdot v_{t,x} \rangle|^2\right) \, \rmd t = 2D_1 t.
\end{align}
It follows then that $t^{-1} M_t$ is a supermartingale, and so by Doob's martingale convergence theorem and the above bound, it converges almost surely to $0$. This clearly implies~\eqref{eq:Lyapunov}, since for any $v \in \mathbb{S}^{d-1}$, we have
\[
\lim_{t\to \infty}\frac{1}{t}\log |v_{t_x}| = D_1 d\, .
\]
This completes the proof.
\end{proof}

\noindent \textbf{Acknowledgements.} We thank G. Eyink, A. Frishman and S. Punshon-Smith for useful discussions.
The research of MCZ was partially supported by the Royal Society URF\textbackslash R1\textbackslash 191492 and
EPSRC Horizon Europe Guarantee EP/X020886/1.
The research of TDD was partially supported by the NSF
DMS-2106233 grant and  NSF CAREER award \#2235395. 
The research of RSG was partially supported by the Deutsche Forschungsgemeinschaft through the SPP 2410/1 \emph{Hyperbolic Balance Laws in Fluid Mechanics: Complexity, Scales, Randomness}.

\bibliographystyle{abbrv}
\bibliography{Kraichnan_biblio.bib}

\end{document}